\RequirePackage{fix-cm}
\documentclass[smallextended]{svjour3}       
\smartqed  
\usepackage{graphicx}
\usepackage{enumerate}
\usepackage{amsmath}
\usepackage{amssymb}
\usepackage{amsfonts}
\usepackage{microtype}
\usepackage{mathrsfs}
\usepackage{graphicx}
\usepackage{color}
\usepackage{latexsym}
\usepackage{rotating}
\usepackage{xspace}
\usepackage[all]{xy}
\usepackage{longtable}
\usepackage{leftidx}
\usepackage{mathtools}

\newtheorem{thm}{Theorem}
\newtheorem{cor}{Corollary}
\newtheorem{lem}{Lemma}

\newtheorem{nota}{Notation}
\newtheorem{dfn}{Definition}
\newtheorem{pro}{Proposition}
\newtheorem{rem}{Remark}

=22cm\headheight=0cm\topskip=0cm\topmargin=0cm

\begin{document}

\title{On dominating graph of graphs, median graphs and partial cubes, and graphs in which complement of every minimal dominating set is minimal dominating 
}
	




\author{Alireza Mofidi} 

\institute{Alireza Mofidi
	\at
	$^1$ Department of Mathematics and Computer Science, Amirkabir University of Technology (Tehran Polytechnic), Iran,
	\and
	$^2$ School of Mathematics, Institute for Research in Fundamental Sciences {\rm (IPM)},
	P.O. Box {\rm 19395-5746}, Tehran, Iran. \\
	\\            
	\email{mofidi@aut.ac.ir}
}


\maketitle

\begin{abstract}
The dominating graph of a graph $G$ is a graph whose vertices correspond to the dominating sets of $G$ and two vertices are adjacent whenever their corresponding dominating sets differ in exactly one vertex.
Studying properties of dominating graph has become an increasingly interesting subject in domination theory.
On the other hand, median graphs and partial cubes are two fundamental graph classes
in graph theory.
In this paper, we make some new connections between
domination theory
and the theory of median graphs and partial cubes.
As the main result,
we show that the following conditions are equivalent for every graph $G \not \simeq C_4$ with no isolated vertex,
and in particular, that the simple third
condition completely characterizes first two ones in which three concepts of dominating graphs, median graphs and complement of minimal dominating sets get related:

\vspace{1mm}

- The dominating graph of $G$ is a median graph,

- The complement of every minimal dominating set of $G$ is a minimal dominating set,

- Every vertex of $G$ is either of degree $1$ or adjacent to a vertex of degree $1$.

\vspace{1mm}

As another result, we prove that the dominating graph of every graph is a partial cube and
also give some examples to show that not all partial cubes or median graphs are isomorphic to the dominating graph of a graph.
The above-mentioned results, as another highlight of the paper, 
provide novel infinite sources of examples of median graphs and partial cubes.

\vspace{3mm}




\end{abstract}

\section{Introduction}\label{sectionintroduction}
In this paper we consider finite, simple and undirected graphs.
For a graph $G$, we say that a subset $D \subseteq V(G)$ {\em dominates} a vertex $v \in V(G)$ if either $v$ belongs to $D$ or is adjacent to some vertex in $D$.
We say that a vertex $v$ dominates a vertex $w$ if $\{v\}$ dominates $w$.
By a {\em dominating set} of $G$, we mean a $D \subseteq V(G)$ that dominates every vertex in $G$.
By a {\em minimal dominating set (MDS)} of $G$, we mean a dominating set $D$ such that for every $x \in D$, the set $D-\{x\}$ is not a dominating set anymore.
By a {\em minimum dominating set}, we mean a dominating set of smallest size in $G$. The size of a minimum dominating set is called the {\em domination number} of $G$. Note that a minimum dominating set is a minimal dominating set, but the converse does not necessarily hold.

\vspace{0.75mm}

The concept of the $k$-dominating graph of a graph was introduced by Haas and Seyffarth in  \cite{HaasSeyffarthThekdominatinggraph} as follows:

\begin{dfn}\label{defkdominatinggraph}
	Let $G$ be a graph and $k$ a positive integer. By the {\em $k$-dominating graph} of $G$, denoted by $D_k(G)$, we mean the graph whose vertices correspond to (and are identified with) the dominating sets of $G$ that have cardinality less than or equal $k$. Also two vertices
	in the graph $D_k(G)$ are adjacent if and only if their corresponding dominating sets in $G$ differ by exactly one element
	(equivalently speaking, their symmetric difference has size equal to $1$).
	We call $D_{|G|}(G)$ the {\em dominating graph} of $G$ and for simplicity,
	if there is no danger of confusion,
	use the notation $D(G)$ for it.
\end{dfn}

Historically, one of the main motivations of the above definition was its application in the
study of the well-known dominating set reconfiguration (DSR) problem. This problem asks whether two given dominating sets $A$ and $B$ of size at most $k$ of a graph, there exists a sequence of dominating sets of the graph starting from $A$ and ending with $B$, such that each dominating set in the sequence is of size at most $k$ and can be obtained from the previous one by either adding or deleting exactly one vertex
(see for instance \cite{SuzukiMouawadNishimuraReconfigurationofdominatingsets} for more information about
the DSR problem).
It is fairly easy to observe that the DSR problem 
can be naturally translated in terms of the connectivity properties of the $k$-dominating graphs by simply asking that
whether two given vertices of $D_k(G)$ belong to the same connected component of $D_k(G)$ or not.
Due to this fact, the connectivity aspects of the dominating graph and also some other aspects of
the DSR problem
became the subject of interest
in papers such as 
\cite{HaasSeyffarthThekdominatinggraph,HaasSeyffarthReconfiguringdominatingsetsinsomewellcoveredandotherclassesofgraphs,MynhardtRouxTeshimaConnectedkdominatinggraphs} and
\cite{SuzukiMouawadNishimuraReconfigurationofdominatingsets}.
Beside connectivity features as one of the initial motivations, several other aspects of the $k$-dominating and dominating graphs
were further explored in various works such as
\cite{AdarichevaBozemanClarkeHaasMessingerSeyffarthSmithReconfigurationgraphsfordominatingsets2021,AlikhaniFatehiKlavzarOntheStructureofDominatingGraphs,HaasSeyffarthReconfiguringdominatingsetsinsomewellcoveredandotherclassesofgraphs,MynhardtRouxTeshimaConnectedkdominatinggraphs,MynhardtNasserasrReconfigurationofcolouringsanddominatingsetsingraphs,SuzukiMouawadNishimuraReconfigurationofdominatingsetsinCoCOON,SuzukiMouawadNishimuraReconfigurationofdominatingsets}.
In fact, the question of studying or characterizing graphs whose dominating graphs possess certain given properties has become an increasingly interesting research topic in domination theory of graphs.
For instance, problems such as characterization of graphs which are isomorphic to their dominating graphs,
or studying graphs which are isomorphic to some dominating graph, or investigation of connected graphs with regular dominating graph, or inquiring the existence of Hamilton paths in the dominating graph, etc,
have all been considered in numerous works including the above-mentioned papers.
Also the study of $k$-dominating and dominating graphs was further
related to similar studies in graph colorings,
independent sets, cliques and vertex covers in several works such as
\cite{FindingpathsbetweengraphcolouringsPSPACEcompletenessandsuperpolynomialdistances,CerecedavandenHeuvelJohnsonFindingpathsbetween3colorings,CerecedavandenHeuvelJohnsonConnectednessofthegraphofvertexcolourings,ItoDemaineHarveyPapadimitriouSideriUeharaUnoOnthecomplexityofreconfigurationproblems,ItoKaminskiDemaineReconfigurationoflistedgecoloringsinagraph}.
It is worth mentioning that a concept similar to dominating graphs but for total dominating sets was worked out in \cite{AlikhaniFatehiMynhardtOnktotaldominatinggraphs}.

\vspace{0.75mm}

Now we talk about the second area of graph theory that this paper deals with.
Median graphs and partial cubes,
as defined below, are two fundamental and central classes of graphs appearing across several major research areas in graph theory and network theory, such as metric graph theory, graph products, graph embedding problems, study of hypercubes, interconnection networks, etc.

\begin{dfn}\label{defMediangraph}
	By a {\em median graph} we mean an undirected graph with the property that every three vertices $a, b$ and $c$ of it have a unique median, where by a median of $a, b$ and $c$, we mean a vertex that belongs to the shortest paths between each pair of $a, b$ and $c$.
\end{dfn}

\begin{dfn}\label{defPartialCube}
	A {\em partial cube} is
	a graph $G$ that can be isometrically embedded in a hypercube,
	meaning that
	$G$ can be identified with a subgraph of a hypercube in such a way that the distance between any two vertices in $G$ is the same as the distance between those two vertices in the hypercube.
	Equivalently,
	a partial cube is a graph $G$ whose vertices can be labeled with subsets of a fix set
	in such a way that the distance between any two vertices in $G$ is equal to the Hamming distance (i.e. the size of the symmetric difference) between their labels.
\end{dfn}

The classes of median graphs and partial cubes have
been studied for more than half a century in several contexts 
and nowadays are still very active areas of research.
The reason that these classes considered central graph classes in graph theory
is due to their appearances in a wide range of areas of mathematics and computer science as well as their very rich combinatorial structure.
Historically, median graphs
arose in universal algebra
in \cite{AvannMetricternarydistributivesemilattices} and \cite{BirkhoffandKissAternaryoperationindistributivelattices}, and then were developed extensively in combinatorics and many other domains. 
To see more details on median graphs and their connections with other discrete and geometric structures,
interested reader can refer to the surveys written by Bandelt-Chepoi 
\cite{BandeltChepoiMetricgraphtheoryandgeometryasurvey},
or Klav\v{z}ar-Mulder \cite{KlavzarMulderMediangraphscharacterizationslocationtheoryandrelatedstructures},
or Chalopin et al. \cite{ChalopinChepoiHiraiOsajdaWeaklymodulargraphsandnonpositivecurvature},
or the books
\cite{bookmediatheory,BookHammackImrichKlavzarHandbookofProductGraphs,BookImrichKlavzarProductGraphs,BookKnuthTheArtofComputerProgramming,bookpartialcubes}
which cover the topic in substantial detail.
Also a few of the earlier books on this area are
\cite{BookFederStableNetworksandProductGraphs,BookMulderTheIntervalFunctionofaGraph,NebeskyMediangraphs}
and
\cite{BookvandeVelTheoryofConvexStructures}.

\vspace{0.75mm}

On the other hand, partial cubes
were introduced by Graham and Pollak in \cite{GrahamPollak} for modeling interconnection networks and were scrutinized widely afterwards in various works such as \cite{Chepoi,Djokovic,EppsteinThelatticedimension,KlavzarShpectorovTribesofcubicpartialcubes,OvchinnikovPartialcubesStructurescharacterizations,WinklerIsometricembeddingsinproductsofcompletegraphs}.
Also in \cite{MofidiOnpartialcubeswellgradedfamiliesandtheirdualswithsomeapplicationsingraphs}, the author of the present paper studied numerous aspects of partial cubes from the point of view of (hypergraph) dualities.
For a comprehensive account of known results on partial cubes, the interested reader can see the books
\cite{bookmediatheory,BookHammackImrichKlavzarHandbookofProductGraphs,BookImrichKlavzarProductGraphs} and \cite{bookpartialcubes}.
Nowadays, the investigation of partial cubes and median graphs has found many applications in different areas of combinatorics and established strong links to other fields of mathematics as well as computer science.
Note that median graphs and partial cubes are closely related to each other. Indeed, every median graph is a partial cube, 
but the converse is not necessarily true since for instance, the cycle graph $C_6$ is a partial cube but is not a median graph.

\vspace{0.75mm}

The third concept that this paper concerns with is the complement of minimal dominating sets.
The idea of studying complements of dominating sets, in particular
when those complements are also dominating sets,
has a long history in
the literature of domination theory and
has been a question of great interest
in many works such as
\cite{domkeDunbarMarkusTheinversedominationnumberofagraph,frendrupHenningOnaconjectureaboutinversedominationingraphs,HenningLowensteinRautenbachRemarksaboutdisjointdominatingsets2009,johnsonPrierWalshHaynesOnaproblemofDomkeDunbarHaynesHedetniemiandMarkusconcerningtheinversedominationnumber,KropMcDonaldPuleoUpperboundsforinversedominationingraphs2021,OreTheoryofgraphs}
and in particular, the thesis
\cite{lowenstein2010complementInthecomplementofadominatingsetPhDThesis}.
It was known (by an observation in \cite{OreTheoryofgraphs}) that the complement of every minimal dominating set in a graph with no isolated vertex is a dominating set. However, in general it might not be a minimal dominating set.

\vspace{0.75mm}

During the course of our investigation in this paper, we study dominating graphs of graphs and make some new connections between domination theory and the theory of median graphs and partial cubes, in particular, between the three above-mentioned research areas in graph theory, namely, dominating graph of graphs and complement of dominating sets, both from domination theory side, and the classes of median graphs and partial cubes. 
We delve into analysing the structure of dominating graphs and consider, roughly speaking, the general problem of characterizing their nice properties from the perspective of median and partial cube properties, and also connect these ideas to the independent subject of the complement of dominating sets. As the main result of the paper, in Theorem \ref{thmDMiffMDScompMDSiffENLOAL},
we completely characterize graphs whose dominating graphs have the celebrated property of being median (we call them DM-graphs), and moreover,
fully characterize (in the same theorem combined with Corollary \ref{characterizationofMDScoMDSgraphs}) graphs with the property that the complement of every minimal dominating set is again a minimal dominating set (we call such graphs MDScoMDS graphs).
More precisely, in this theorem, we show that the dominating graph of a graph $G \not \simeq C_4$ with no isolated vertex is a median graph if and only if the complement of every minimal dominating set of graph is a minimal dominating set, 
and furthermore, both of these conditions are equivalent to and are completely characterized by a simple third condition saying that
every vertex of graph is either a leaf (of degree $1$) or adjacent to a leaf. 

\vspace{0.75mm}

As another highlight of the paper, in Theorem \ref{Dominatinggraphsarepartialcube}, we prove that the dominating graph of every graph is a partial cube (and also give some examples to show that not all partial cubes or median graphs are isomorphic to the dominating graph of a graph).
Finding new sources of examples of median graphs and partial cubes
has always been a topic of interest and was
considered extensively 
in the literature (see references on median graphs and partial cubes cited 
earlier).
The above-discussed results of this paper in theorems \ref{thmDMiffMDScompMDSiffENLOAL} and \ref{Dominatinggraphsarepartialcube}, beside characterizing the mentioned properties in graphs, furthermore, as another high point of the paper, provide us with some new infinite sources of examples of median graphs and partial cubes since, as mentioned, the dominating graph of every graph 
in which any vertex is either leaf or adjacent to a leaf gives rise to a median graph, and also the dominating graph of every graph gives rise to a partial cube.
Also recall that these results, as a consequence, help to elaborate relationships between domination theory and metric graph theory, since median graphs and partial cubes which get related to dominating graphs in these theorems, constitute two of the most important graph classes in metric graph theory.

\vspace{0.75mm}

Finally, it should be mentioned that as a corollary of the 
aforestated
characterizations of Theorem \ref{thmDMiffMDScompMDSiffENLOAL}, we make another connection to the
subject of the complement of dominating sets by concluding
a relationship between
median property of dominating graphs and the well-known notion of inverse domination number (see
\cite{domkeDunbarMarkusTheinversedominationnumberofagraph,frendrupHenningOnaconjectureaboutinversedominationingraphs,johnsonPrierWalshHaynesOnaproblemofDomkeDunbarHaynesHedetniemiandMarkusconcerningtheinversedominationnumber,KulliSigarkantiInversedominationingraphs1991}
and
\cite{lowenstein2010complementInthecomplementofadominatingsetPhDThesis}) of graphs (i.e. the minimum cardinality of a dominating set whose complement contains a minimum dominating set of the graph).
In particular, in Corollary \ref{inDMgraphsinversedominationnumberisequalnminusdominationnumber}, we conclude that for every graph $G$ with $n$ vertices (all non-isolated) satisfying either of the
conditions of clauses of Theorem \ref{thmDMiffMDScompMDSiffENLOAL} (recall that those conditions are proved that are equivalent except for the case $G \simeq C_4$), we have $\gamma(G)+\gamma^{-1}(G)=n$ where $\gamma^{-1}(G)$ is the inverse domination number of $G$.
It is worth noting that the class of graphs $G$
with the property $\gamma(G)+\gamma^{-1}(G)=n$ was considered in \cite{domkeDunbarMarkusTheinversedominationnumberofagraph}.
Corollary \ref{inDMgraphsinversedominationnumberisequalnminusdominationnumber}, clarifies the relationship between that class and graphs studied in this paper. 

\vspace{0.75mm}

Organization of the rest of the paper is as follows. After defining necessary concepts and notations and reviewing required notions in Subsection \ref{subsectiondefinitionsandnotations}, in Section \ref{sectionpreliminaryresults}, we will prove several technical results that we will need to establish our main results.
In Section \ref{MainResults}, we will state and prove main results of the paper which are Theorem \ref{thmDMiffMDScompMDSiffENLOAL} and a few corollaries of it and also Theorem \ref{Dominatinggraphsarepartialcube}.

\subsection{Some definitions and notations}\label{subsectiondefinitionsandnotations}

Let $G$ be a simple undirected graph.
For every $A \subseteq V(G)$, we denote the induced subgraph from $G$ on $A$ by $\langle A \rangle$.
By a {\em leaf} in $G$ we mean a vertex of degree $1$. A vertex which is not a leaf is called a {\em non-leaf} vertex.
For a vertex $v$ of $G$, by $N_G(v)$, we mean the set of vertices adjacent to $v$ and also we let $N_G[v]:=N_G(v) \cup \{v\}$ (when clear from the context, we instead use simpler notations $N(v)$ and $N[v]$ respectively).
By an {\em independent set} of $G$, we mean a set of vertices that there is no edge between any two of them.
By a {\em maximal independent set (MIS)}, we mean an independent set which is not the strict subset of any other independent set.
For a dominating set $D$ of $G$ and $v \in D$, we call a vertex $w \in V(G)$ a {\em $D$-private neighbor} (or simply a private neighbor) of $v$ if $w$ is dominated by $v$ but is not dominated by any vertex in $D-\{v\}$.
So a vertex can be $D$-private neighbor of itself.
It is straightforward to verify that every vertex in any given minimal dominating set $D$ must have a private neighbor.
By $Dom(G)$ we mean the family of all dominating sets of $G$.
Also by $mDom(G)$ we mean the family of all minimal dominating sets of $G$.

\vspace{1mm}

We define the following notions which are the central definitions of this paper.

\begin{dfn}\label{defDominatingMedianGraphs}
	We call a graph $G$ {\em dominating-median (or DM for short)} if its dominating graph $D_{|G|}(G)$ is a median graph.
\end{dfn}	

\begin{dfn}\label{defMDScoMDSGraphs}
	We call a graph $G$ a {\em MDScoMDS} graph if the complement of every minimal dominating set of $G$ is also a minimal dominating set.
\end{dfn}

Recall that the idea of studying complements of dominating sets (and also pairs of disjoint dominating sets) has been investigated in several papers in the literature.
In this direction, the following notion was introduced in \cite{KulliSigarkantiInversedominationingraphs1991} and worked out in several works such as
\cite{domkeDunbarMarkusTheinversedominationnumberofagraph,frendrupHenningOnaconjectureaboutinversedominationingraphs,johnsonPrierWalshHaynesOnaproblemofDomkeDunbarHaynesHedetniemiandMarkusconcerningtheinversedominationnumber}
and
\cite{lowenstein2010complementInthecomplementofadominatingsetPhDThesis}.

\begin{dfn}\label{dfninversedominationnumberofgraph}
	Let $G$ be a graph with no isolated vertex and $D$ a minimum dominating set of $G$.
	We call a set $B \subseteq V(G)$ an {\em inverse dominating set} with respect to the dominating set $D$ if $B \subseteq V(G)-D$ and $B$ is a dominating set of $G$.
	We call a subset of vertices of $G$ an inverse dominating set in $G$, if it is an inverse dominating set with respect to some minimum dominating set of $G$.
	The {\em inverse domination number} of $G$, denoted by $\gamma^{-1}(G)$, is the minimum cardinality of an inverse dominating set in $G$.
\end{dfn}

As mentioned earlier, it is known (by an observation of Ore \cite{OreTheoryofgraphs}) that in a graph with no isolated vertex, the complement of every minimal dominating set is a dominating set (and so the complement of every minimum dominating set is also a dominating set).
It follows that any graph with no isolated vertex contains an inverse dominating set and so the inverse domination number is defined for it.
Note that although the complement of a minimal dominating set in such graphs is a dominating set, it might not necessarily be a minimal dominating set. For instance, every vertex of $C_3$ is a minimal dominating set while its complement is not minimal dominating set.

For additional fundamental concepts from the domination theory in graphs see \cite{BookHaynesHedetniemiSlaterFundamentalsofDominationinGraphs}.
Also interested reader can refer to the books
\cite{BookHaynesHedetniemiHenningTopicsindominationingraphs} and
\cite{BookHaynesHedetniemiHenningStructuresofdominationingraphs}
to see many of the most significant recent developments
in domination theory in graphs and also several variants of the notion of domination.

We recall that every median graph is a partial cube but the converse does not hold.
Partial cubes and median graphs are connected graphs.
The {\em corona product} $G \odot H$ of two graphs $G$ and $H$ is the graph obtained from $G$ and $H$ by taking one copy of $G$ and $|V(G)|$ copies of $H$ and then joining by an edge every vertex from the $i$th-copy of $H$ with the $i$th-vertex of $G$ (for every $i \leqslant |V(G)|$).

\vspace{1mm}

We use the following notation frequently in the paper.

\begin{nota}
	For every three sets $A$, $B$ and $C$, we define
	$$*(A,B,C):=(A \cap B) \cup (A \cap C) \cup (B \cap C).$$
\end{nota}

\section{Some preliminary results and technical lemmas}\label{sectionpreliminaryresults}

The main results of this paper, in particular Theorem \ref{thmDMiffMDScompMDSiffENLOAL} and Theorem \ref{Dominatinggraphsarepartialcube}, will appear in the next section and will characterize, most importantly, DM graphs
as well as MDScoMDS graphs (see Definitions \ref{defDominatingMedianGraphs} and \ref{defMDScoMDSGraphs} above) and demonstrate relation between dominating graphs of graphs, complement of minimal dominating sets, median graphs and partial cubes.
In order to establish those results, we first need to prove several preliminary results and technical lemmas in this section that will be used in the proofs of our main theorems.

\begin{pro}\label{DMiff*everythreeDSisDS}
	Let $G$ be a graph and $H:=D_{|G|}(G)$ be its dominating graph. Then, the following hold.
	\begin{enumerate}
		\item{$G$ is a DM graph (remind that it means that $H$ is a median graph).}\label{domgraphismedian}

		\item{For every three dominating sets $D_1, D_2$ and $D_3$ of $G$, the set $D_4:=*(D_1,D_2,D_3)$ is a dominating set.}\label{*ofeverythreeDSisDS}
		
		\item{For every three minimal dominating sets $D_1, D_2$ and $D_3$ of $G$, the set $D_4:=*(D_1,D_2,D_3)$ is a dominating set.}\label{*ofeverythreeminDSisDS}
	\end{enumerate}
\end{pro}
\begin{proof}
	$(\ref{domgraphismedian} \Rightarrow \ref{*ofeverythreeDSisDS})$
	Assume that $H$ is a median graph and $D_1, D_2$ and $D_3$ are dominating sets of $G$.
	So there exists a vertex $A$ in $H$ lying on some three shortest path's $P_1, P_2$ and $P_3$ from $D_1$ to $D_2$, from $D_1$ to $D_3$, and from $D_2$ to $D_3$ respectively.
	It is not hard to see that any shortest path from $D_1$ to $D_2$ in $H$ is of size $|D_1 \triangle D_2|$ (note that the construction of such a shortest path is also described in the proof of Theorem \ref{Dominatinggraphsarepartialcube}),
	and moreover, for every vertex $S$ on any shortest path from $D_1$ to $D_2$, we must have $S \subseteq D_1 \cup D_2$.
	Therefore, $A \subseteq D_1 \cup D_2$.
	Similarly, $A \subseteq D_1 \cup D_3$ and $A \subseteq D_2 \cup D_3$.
	It follows that
	$$A \subseteq (D_1 \cup D_2) \cap (D_1 \cup D_3) \cap (D_2 \cup D_3) = (D_1 \cap D_2) \cup (D_1 \cap D_3) \cup (D_2 \cap D_3) =*(D_1,D_2,D_3).$$
	Hence, since $A$ is a dominating set, $*(D_1,D_2,D_3)$ is a dominating set too.
	
	\vspace{1.5mm}	
	
	$(\ref{*ofeverythreeDSisDS} \Rightarrow \ref{domgraphismedian})$
	Assume that $D_1, D_2$ and $D_3$ are dominating sets of $G$.
	We claim that $D_4$ is the unique median of the three vertices $D_1, D_2$ and $D_3$ in graph $H$.
	We define a path $P_1$ in $H$ between $D_1$ and $D_2$ as follows:

	\begin{enumerate}
		\item{Start from $D_1$ and add to it one by one (according to an arbitrary ordering) every vertex of $(D_2 \cap D_3)-D_1$ until reaching to $S_1:=D_1 \cup (D_2 \cap D_3)$. Since every set in this process contains $D_1$, they are all dominating sets.}
		
		\item{Start from $S_1$ and remove from it one by one (according to an arbitrary ordering) every vertex of $D_1 - (D_2 \cup D_3)$ until reaching to $S_2:=D_4$ (note that $S_1-(D_1 - (D_2 \cup D_3))=D_4$ and recall that $D_4=*(D_1,D_2,D_3)$). Obviously, every set in this process contains $D_4$ and since $D_4$ is assumed to be a dominating set, they are all dominating sets.}
		
		\item{Start from $S_2$ and add to it one by one (according to an arbitrary ordering) every vertex of $D_2-D_4$ until reaching to $S_3:=D_2 \cup (D_1 \cap D_3)$ (note that $D_2 \cup (D_1 \cap D_3)=D_2 \cup D_4$). Since every set in this process contains $D_4$, they are all dominating sets.}
		
		\item{Start from $S_3$ and remove from it one by one (according to an arbitrary ordering) every vertex of $(D_1 \cap D_3)-D_2$ until reaching to $S_4:=D_2$. Since every set in this process contains $D_2$, they are all dominating sets.}
	\end{enumerate}
	
	Since every set obtained in above steps of moving from $D_1$ to $D_2$ on the path $P_1$ is a dominating set, $P_1$ is a path in $H$ connecting $D_1$ to $D_2$. Also it is not hard to verify that the length of $P_1$ is $|D_1 \triangle D_2|$ which means that it is a shortest path in $H$ between $D_1$ and $D_2$. Also it contains $D_4$.

	Similar to above, we can define some shortest path $P_2$ from $D_1$ to $D_3$ and some shortest path $P_3$ from $D_2$ to $D_3$ such that both $P_2$ and $P_3$ contain $D_4$.
	Now $D_4$ is a vertex lying on above shortest path's between pairs of the vertices $D_1$, $D_2$ and $D_3$.
	Therefore, $D_4$ is a median of the three vertices $D_1$, $D_2$ and $D_3$ in $H$.

	Now we show the uniqueness (in the definition of median graphs) of the vertex with this median property.
	Assume that $D_5$ is a vertex of $H$ lying on some shortest path's between each pair of the vertices $D_1$, $D_2$ and $D_3$.
	By what was shown, the length of any shortest path between $D_1$ and $D_2$ is equal to $|D_1 \triangle D_2|$. So, since $D_5$ lies on some shortest path between $D_1$ and $D_2$, it is not hard to see that we must have $D_1 \cap D_2 \subseteq D_5 \subseteq D_1 \cup D_2$.
	Similarly, $D_1 \cap D_3 \subseteq D_5 \subseteq D_1 \cup D_3$ and $D_2 \cap D_3 \subseteq D_5 \subseteq D_2 \cup D_3$.
	It follows that
	$$*(D_1,D_2,D_3)=(D_1 \cap D_2) \cup (D_1 \cap D_3) \cup (D_2 \cap D_3) \subseteq D_5$$ $$\subseteq (D_1 \cup D_2) \cap (D_1 \cup D_3) \cap (D_2 \cup D_3) = *(D_1,D_2,D_3).$$
	So $D_5=*(D_1,D_2,D_3)=D_4$ which implies the uniqueness of median.
	Therefore, $H$ is a median graph.
	
	\vspace{1.5mm}
	
	$(\ref{*ofeverythreeDSisDS} \Leftrightarrow \ref{*ofeverythreeminDSisDS})$ The direction $\ref{*ofeverythreeDSisDS} \Rightarrow \ref{*ofeverythreeminDSisDS}$ is obvious. For the converse, assume that $D_1, D_2$ and $D_3$ are some dominating sets. So there are minimal dominating sets $D'_1, D'_2$ and $D'_3$ such that $D'_1 \subseteq D_1$, $D'_2 \subseteq D_2$ and $D'_3 \subseteq D_3$.
	Now, by the assumption in $\ref{*ofeverythreeminDSisDS}$, $*(D'_1,D'_2,D'_3)$ is a dominating set. But it is easy to see that $*(D'_1,D'_2,D'_3) \subseteq *(D_1,D_2,D_3)$. Hence, $*(D_1,D_2,D_3)$ is a dominating set too. It follows the result.
	\hfill $\square$
\end{proof}

\begin{lem}\label{Conectedcomponents}
	Let $G$ be a graph with connected components $G_1,\dots G_r$. Then, the following hold.
	\begin{enumerate}
		\item{We have
			$$Dom(G)=\{D_1 \cup \dots \cup D_r: D_i \in Dom(G_i), \ for \ each \ i \leqslant r\},$$
			$$mDom(G)=\{D_1 \cup \dots \cup D_r: D_i \in mDom(G_i), \ for \ each \ i \leqslant r\}.$$}\label{DomandmDomofconnectedcomps}
		\item{$G$ is DM if and only if every connected component of $G$ is a DM graph.}\label{AgraphisDMiffeveryconnectedcomponentisDM}
		\item{$G$ is MDScoMDS if and only if every connected component of $G$ is a MDScoMDS graph.}\label{AgraphisMDScompMDSiffeveryconnectedcomponentisMDScompMDS}
	\end{enumerate}
\end{lem}
\begin{proof}
	\ref{DomandmDomofconnectedcomps}) is easy.
	
	\ref{AgraphisDMiffeveryconnectedcomponentisDM})
	Assume that each component $G_i$ is a DM graph. Let $A, B$ and $C$ be some arbitrary dominating sets of $G$. By Proposition $\ref{DMiff*everythreeDSisDS}$, 
	it is enough to show that $*(A,B,C)$ is a dominating set of $G$. It is not hard to see that $*(A,B,C)=\bigcup_{j \leqslant r} *(A_i,B_i,C_i)$ where $A_i:=A \cap V(G_i)$, $B_i:=B \cap V(G_i)$ and $C_i:=C \cap V(G_i)$ for each $i \leqslant r$. Also by Part \ref{DomandmDomofconnectedcomps}, $A_i$, $B_i$ and $C_i$ are dominating sets of $G_i$ for each $i$. Again, by Proposition $\ref{DMiff*everythreeDSisDS}$, 
	$*(A_i,B_i,C_i)$ is a dominating set of $G_i$ for each $i$. Hence, by Part \ref{DomandmDomofconnectedcomps}, $*(A,B,C)$ is a dominating set of $G$. It follows that $G$ is a DM graph.
	
	Now we prove the converse. Assume that $G$ is DM. By using Proposition $\ref{DMiff*everythreeDSisDS}$, it would be enough to show that for a fixed arbitrary $i_0 \leqslant r$, and every three dominating sets $A_{i_0}$, $B_{i_0}$ and $C_{i_0}$ of $G_{i_0}$, the set $*(A_{i_0}, B_{i_0}, C_{i_0})$ is a dominating set of $G_{i_0}$.
	For every $i \leqslant r$ with $i \not = i_0$ define $A_i$, $B_i$ and $C_i$ to be all equal $V(G_i)$.
	Also let $A:= \bigcup_{i \leqslant r} A_i$, $B:= \bigcup_{i \leqslant r} B_i$ and $C:= \bigcup_{i \leqslant r} C_i$.
	Obviously, by Part \ref{DomandmDomofconnectedcomps}, $A, B$ and $C$ are dominating sets of $G$.
	Since $G$ is DM, by Proposition $\ref{DMiff*everythreeDSisDS}$, the set $*(A,B,C)$ is a dominating set of $G$.
	It is also easily seen that $*(A,B,C) \cap V(G_{i_0})=*(A_{i_0},B_{i_0},C_{i_0})$.
	Consequently, by Part \ref{DomandmDomofconnectedcomps}, $*(A_{i_0},B_{i_0},C_{i_0})$ is a dominating set of $G_{i_0}$. It follows the result.	
	
	\ref{AgraphisMDScompMDSiffeveryconnectedcomponentisMDScompMDS})
	By using Part \ref{DomandmDomofconnectedcomps}, it is not difficult to verify that the complement of every minimal dominating set in $G$ is a minimal dominating set if and only if the same happens in each connected component of $G$. It follows that $G$ is MDScoMDS if and only if every connected component of it is MDScoMDS.
	\hfill $\square$
\end{proof}

\begin{lem}\label{addingcoleafpreservesDM}
	Let $G$ be a graph and $v$ a leaf vertex of it which is adjacent to a vertex $w$.
	Let $H$ be the graph obtained from $G$ by adding a number of new vertices to $G$ and connecting them only to $w$ (so, they would be leaf vertices in $H$).
	Then, $H$ is DM if and only if $G$ is DM.
\end{lem}
\begin{proof}
	For simplicity, we assume that only one new leaf $z$ is added to $G$ (obviously the general case would be obtained by repeating this).
	Assume that $G$ is DM and $A, B$ and $C$ are three arbitrary dominating sets of $H$.
	In order to show that $H$ is DM, it is enough (By Proposition \ref{DMiff*everythreeDSisDS}) to show that $*(A,B,C)$ is a dominating set of $H$.
	Since $v$ is leaf of $H$, $A$ must contain at least one of $v$ or $w$.
	So, it is not hard to verify that $A':=A-\{z\}$ is a dominating set of $G$.
	Similarly, $B':=B-\{z\}$ and $C':=C-\{z\}$ are dominating sets of $G$.
	Since $G$ is DM, by Proposition \ref{DMiff*everythreeDSisDS}, $*(A',B',C')$ is a dominating set of $G$.
	Thus, since $*(A',B',C') \subseteq *(A,B,C)$, the set $*(A,B,C)$ (which is a subset of vertices of $H$) dominates every vertex in $V(G)$.
	If $w \in *(A,B,C)$, then $*(A,B,C)$ would be a dominating set of $H$. Otherwise, if $w \not \in *(A,B,C)$, then at least two of $A, B$ and $C$, say $A$ and $B$, do not contain $w$.
	So, since $A$ and $B$ are dominating sets of $H$, they must contain $z$ and $v$. It follows that $z \in *(A,B,C)$. Therefore, $*(A,B,C)$ is a dominating set of $H$.

	For the converse, assume that $H$ is DM and $X, Y$ and $Z$ are three dominating sets of $G$. Again, it would be enough
	(by Proposition \ref{DMiff*everythreeDSisDS}) to show that $*(X,Y,Z)$ is a dominating set of $G$.
	Let $A:=X \cup \{z\}$, $B:=Y \cup \{z\}$ and $C:=Z \cup \{z\}$.
	Obviously, $A, B$ and $C$ are dominating sets of $H$. So, by Proposition \ref{DMiff*everythreeDSisDS}, $*(A,B,C)$ is a dominating set of $H$.
	It is easily seen that $*(A,B,C)-\{z\} = *(X,Y,Z)$.
	Since $z$ is a leaf and is adjacent to $w$, it is sufficient to show that $*(X,Y,Z)$ dominates $w$.
	If $w$ belongs to at least two of $X,Y$ and $Z$, then it belongs to $*(X,Y,Z)$. 
	Otherwise, $w$ does not belong to at least two of $X, Y$ and $Z$, say $X$ and $Y$. Hence, since $X$ and $Y$ dominate $G$, then we must have $v \in X$ and $v \in Y$. Thus, $v \in *(X,Y,Z)$. It follows that $*(X,Y,Z)$ dominates $w$ and we are done.
	\hfill $\square$
\end{proof}

\begin{lem}\label{DMpreservesafteraddingleafstoavertexwhoseneighbshavesomeleafneighb}
	Let $G$ be a graph and $w$ a vertex of it. Assume that every
	vertex in $N_G(w)$ is adjacent to some vertex
	(different from $w$)
	which is a leaf. Let $H$ be the graph obtained from $G$ by adding some new vertices $v_1,\ldots,v_t$ and the edges $wv_1,\ldots, wv_t$ to it. Then, if $G$ is a DM graph, then $H$ is also a DM graph.
\end{lem}
\begin{proof}
	One easily verifies that by Lemma \ref{addingcoleafpreservesDM}, it would be sufficient to prove for the case $t=1$, which means that just a single vertex $v$ and an edge $wv$ is added to $G$.
	Let $A, B, C$ be three arbitrary dominating sets of $H$. By Proposition \ref{DMiff*everythreeDSisDS}, it would be enough to show that $*(A,B,C)$ is also a dominating set of $H$.
	Define $A':=(A-\{v\}) \cup \{w\}$, $B':=(B-\{v\}) \cup \{w\}$ and $C':=(C-\{v\}) \cup \{w\}$. Obviously, $A',B'$ and $C'$ are subsets of $V(G)$.
	
	\vspace{1mm}
	
	{\bf \em Claim 1:} $A'$, $B'$ and $C'$ are all dominating sets of $G$.
	
	{\em Proof of Claim 1:}	
	Assume for contradiction that for example $A'$ does not dominate $G$. So there is $x \in V(G)$ which does not belong to and is not adjacent to any vertex in $A'$. Obviously, $x \not= w$ and so $x$ is not adjacent to $v$ in $H$. Thus, $x$ is not
	dominated by
	$A' \cup \{v\}$ in $H$ and therefore, is not
	dominated by $A$ since $A \subseteq A' \cup \{v\}$.
	It follows that $A$ does not dominate $H$ which is a contradiction.
	So, $A'$ is a dominating set of $G$.
	Similarly, $B'$ and $C'$ are also dominating sets of $G$.
	\hfill {\em Claim 1} $\square$
	
	\vspace{1mm}
	
	{\bf \em Claim 2:} We have $*(A',B',C') \subseteq *(A,B,C) \cup \{w\}$. Moreover, at least one of $v$ and $w$ belongs to $*(A,B,C)$.
	
	{\em Proof of Claim 2:} We have
	$$*(A,B,C) \cup \{w\} = ((A \cup \{w\}) \cap (B \cup \{w\})) \cup (A \cup \{w\}) \cap (C \cup \{w\})) \cup (B \cup \{w\}) \cap (C \cup \{w\}))$$
	$$\supseteq (A' \cap B') \cup (A' \cap C') \cup (B' \cap C')=*(A',B',C').$$
	Since $A, B$ and $C$ are dominating sets of $H$, each one must contain at least one of $v$ or $w$ since $v$ is a leaf in $H$. Therefore, at least one of $v$ or $w$ belong to at least two of $A, B$ and $C$. It follows that $*(A,B,C)$ contains at least one of $v$ or $w$.
	\hfill {\em Claim 2} $\square$
	
	\vspace{1.5mm}
	
	Since as was shown in Claim 1, $A',B'$ and $C'$ are dominating sets of $G$, and $G$ is assumed to be a DM graph, then by Proposition \ref{DMiff*everythreeDSisDS}, $*(A',B',C')$ is a dominating set of $G$.
	Now since by Claim 2 we have $*(A',B',C') \subseteq *(A,B,C) \cup \{w\}$, then the set $*(A,B,C) \cup \{w\}$ dominates every vertex in $V(G)$.
	Hence, $*(A,B,C)$ dominates every vertex in $V(G)-N_G[w]$.
	Also as was shown in Claim 2, $*(A,B,C)$ contains at least one of $v$ or $w$.
	It follows that every vertex in $V(H)-N_G(w)$ is dominated by $*(A,B,C)$.
	On the other hand, by the assumption, every $z \in N_G(w)$ is adjacent to some vertex $z'$ (different from $w$) such that $z'$ is a leaf vertex in $G$ (so $z'$ would be a leaf of $H$ too). Since each of $A, B$ and $C$ is a dominating set of $H$ and $z'$ is a leaf, each of $A, B$ and $C$ must contains at least one of $z$ or $z'$. It follows that at at least one of $z$ or $z'$ belongs to at least two of $A, B$ and $C$. It implies that $*(A,B,C)$ contains at least one of $z$ or $z'$.
	Therefore, every $z \in N_G(w)$ is either inside or adjacent to some vertex of $*(A,B,C)$ in $H$.
	Putting all above together, $*(A,B,C)$ is a dominating set of $H$.
	\hfill $\square$
\end{proof}

\begin{rem}\label{everygraphwithatleastoneedgehasatleasttwoMDS}
	Every graph with at least one edge has at least two distinct minimal dominating sets.
\end{rem}
\begin{proof}
	Assume that $G$ has an edge $uv$ where $u, v \in V(G)$. Let $M$ be a maximal independent set of vertices containing $u$. Then, it is easy to see that $M$ is a minimal dominating set of $G$. Also since $u \in M$ and $u$ is adjacent to $v$, we have $v \not \in M$.
	Similarly, we can find another minimal dominating set containing $v$ and not containing $u$. It follows the result.
	\hfill $\square$
\end{proof}

\section{Main results}\label{MainResults}	

\subsection{Characterization of dominating-median (DM) graphs and MDScoMDS graphs}
The following theorem, which is the main result of this paper, makes novel connections between 
dominating graphs of graphs, complement of dominating sets 
and median graphs. 
More precisely, this result characterizes dominating-median (DM) graphs (see Definition \ref{defDominatingMedianGraphs}) and moreover, characterizes graphs with the property that the complement of every minimal dominating set is a minimal dominating set (MDScoMDS graphs).
In particular, it shows the equivalency between these two conditions and also by a simple third condition, gives a complete characterization for both of them.
Furthermore, it provides some new infinite sources of examples of median graphs since, as we will see, the dominating graph of every graph with the property that every vertex is either leaf or adjacent to a leaf gives rise to a median graph.
As mentioned in the introduction of the paper, finding new sources of examples of median graphs has always been a subject of great interest and was considered extensively in the literature.

Recall from Section \ref{sectionintroduction} that it is known
that the complement of every minimal dominating set in a graph with no isolated vertex is always a dominating set but not necessarily minimal dominating.

\begin{thm}\label{thmDMiffMDScompMDSiffENLOAL}(Main result)
	Let $G$ be a graph with no isolated vertex and $G \not \simeq C_4$. Then, the following are equivalent.
	\begin{enumerate}
		\item{$G$ is a DM graph (recall that it means that the dominating graph $D_{|G|}(G)$ of $G$ is a median graph).}\label{graphisDM}
		\item{$G$ is a MDScoMDS graph (recall that it means that the complement of every minimal dominating set of $G$ is a minimal dominating set).}\label{MDScompMDS} 
		\item{Every vertex of $G$ is either a leaf or adjacent to a leaf.}\label{ENLOAL}
	\end{enumerate}
\end{thm}
\begin{proof}
	$(\ref{graphisDM} \Rightarrow \ref{ENLOAL})$
	Assume for contradiction that there is a vertex $v$ of degree at least two which is not adjacent to any leaf.
	By using Proposition \ref{DMiff*everythreeDSisDS}, it would be enough to show that there are dominating sets $A, B$ and $C$ of $G$ such that $*(A,B,C)$ is not a dominating set.
	Let $H:=N(v)$ and recall that, for simplicity, we use notations $N(v)$ and $N[v]$ instead of $N_G(v)$ and $N_G[v]$ respectively.
	Note that $|H| \geqslant 2$ since $d(v) \geqslant 2$. Also every vertex in $H$ has degree at least two.
	Let $L:=V(G) -N[v]$.
	Let $H_1$ be the set of vertices in $H$ which have some adjacent vertex in $L$.
	Also let $H_2:=H -H_1$.
	We consider the following three cases and in each case find our three suitable dominating sets $A, B$ and $C$ mentioned above.
	
	\vspace{0.75mm}
	
	{\bf Case I:}: In this case we assume that $H_2=\emptyset$.
	So every vertex in $H$ is adjacent to some vertex in $L$.
	Let $w_1$ and $w_2$ be two distinct vertices in $H$.
	Define $A:=L \cup \{v\}$, $B:=L \cup \{w_1\}$ and $C:=L \cup \{w_2\}$.
	It is easy to see that $A, B$ and $C$ are all dominating sets of $G$ and also $*(A,B,C)=L$.
	Thus, $v$ is neither inside nor adjacent to any vertex in $*(A,B,C)$ which follows that $*(A,B,C)$ is not a dominating set of $G$.
	This finishes Case I.
	
	\vspace{0.75mm}
	
	{\bf Case II:}
	In this case we assume that $H_2 \not = \emptyset$ and also every vertex in $H_2$ is adjacent to a vertex in $H_1$ (which implies that $H_1 \not = \emptyset$ and so $L \not = \emptyset$).
	Now define
	$A:=L\cup \{v\}$, $B:=L \cup H_1$ and $C:=L \cup H_2$.
	It is easily seen that $A$, $B$ and $C$ are dominating sets of $G$.
	On the other hand, $*(A,B,C)=L$ which is not a dominating set of $G$. This finishes Case II.
	
	\vspace{0.75mm}
	
	{\bf Case III:}
	In this case we assume that $H_2 \not = \emptyset$ and also there is some vertex $v \in H_2$ which is not adjacent to any vertex in $H_1$ (this case includes the situation that $H_1=\emptyset$).
	So, every such vertex $v$ must be adjacent to some other vertex from $H_2$ since every vertex in $H$ has degree at least two.
	Thus, there is some edge $e$ in $\langle H_2 \rangle$, the induced graph on $H_2$.
	Therefore, by Remark $\ref{everygraphwithatleastoneedgehasatleasttwoMDS}$, the subgraph $\langle H_2 \rangle$ has at least two distinct minimal dominating sets $M_1$ and $M_2$.
	Also, notice that if $H_1 \not = \emptyset$, then $L \not = \emptyset$ and $L$ dominates every vertex in $H_1$.
	Now define $A:=L \cup \{v\}$, $B:=L\cup M_1$ and $C:=L \cup M_2$.
	It is straightforward to verify that $A, B$ and $C$ are dominating sets for $G$ and also
	$*(A,B,C)=L \cup (M_1 \cap M_2)$.
	On the other hand, since $M_1$ and $M_2$ are distinct minimal dominating sets of $\langle H_2 \rangle$, the set $M_1 \cap M_2$ is not a dominating set of the subgraph $\langle H_2 \rangle$. Recall that no vertex of $\langle H_2 \rangle$ is adjacent to any vertex of $L$. Therefore, $*(A,B,C)$ does not dominate all vertices of $\langle H_2 \rangle$. It follows that $*(A,B,C)$ is not a dominating set of $G$. This finishes Case III.

	\vspace{2mm}
	
	$(\ref{ENLOAL} \Rightarrow \ref{graphisDM})$
	By Lemma \ref{Conectedcomponents}(\ref{AgraphisDMiffeveryconnectedcomponentisDM}), it would be enough to prove the statement for the case that $G$ is connected. We prove the statement in the following steps.
	
	\vspace{1mm}
	
	{\bf \em Step I:}
	The statement holds for the case that $G$ is a tree.
	
	{\em Proof of Step I:}
	Since in this step we consider trees, for notational convenience, we use the notation $T$ instead of $G$ for our graphs in this step.
	We prove by induction on the number of the vertices of the tree.
	For the base case of tree with two vertices $P_2$, statement easily holds. Assume that the statement holds for every tree with at most $n$ vertices ($n \geqslant 2$). Let $T$ be a tree with $n+1$ vertices such that every vertex of it is either a leaf or adjacent to a leaf. Choose an arbitrary leaf vertex $v_0$ of $T$ and make $T$ rooted with vertex $v_0$ as the root.
	Let $s$ be the a farthest vertex to $v_0$ in $T$. So $s$ is a leaf vertex. Let $z$ be the parent vertex of $s$.
	Obviously, $z \not = v_0$ since otherwise $T \simeq P_2$ which is impossible since $T$ has at least $3$ vertices.
	It is easily seen that every child of $z$ is a leaf since otherwise $s$ wouldn't be the farthest vertex to $v_0$.
	Let $T'$ be the tree obtained from $T$ by removing all children of $z$ and edges adjacent to them.
	Obviously, $T'$ has at most $n$ vertices and easily seen that has at least $2$ vertices.
	Since by assumption, every vertex of $T$ is either a leaf or adjacent to a leaf, $T'$ has this property too.
	So by induction hypothesis, $T'$ is a DM graph.
	Now we consider $T$ again. Let $h$ be the parent vertex of $z$ in $T$.
	If $h=v_0$, then it is easily seen that $T$ would be the star tree with the central vertex $z$ and at least two leaves, and also $T'$ would be the path $P_2$. In this case, by applying Lemma \ref{addingcoleafpreservesDM}, we conclude that $T$ is a DM graph.
	If $h \not =v_0$, then $h$ would not be a leaf vertex of $T$.
	Since by assumption, every non-leaf vertex of $T$ is adjacent to a leaf,
	$h$ must be adjacent to a leaf vertex in $T$, say $g$. Note that $g \not= z$ since $z$ is not a leaf in $T$ since has at least one child $s$. Also $g \in V(T')$ since clearly $g$ is not a child of $z$.
	Therefore, $z$ and $g$ are two leaf vertices adjacent to $h$ in $T'$.
	Now by Lemma \ref{DMpreservesafteraddingleafstoavertexwhoseneighbshavesomeleafneighb}, $T$ is a DM graph since it can be obtained from $T'$ by adding a few leaves to $z$ (namely, the children of $z$ in $T$).
	This finishes the induction and follows the result of Step I.
	\hfill {\em Step I} $\square$
	
	\vspace{1mm}
	
	{\bf \em Step II:} The statement holds for every graph.
	
	{\em Proof of Step II:}
	As mentioned above, we may assume $G$ to be connected.
	By Proposition \ref{DMiff*everythreeDSisDS}, it would be enough to show that for every three dominating sets $A, B$ and $C$ of $G$, the set $*(A,B,C)$ is a dominating set of $G$.
	Let $T$ be some fixed arbitrary spanning tree of $G$.
	Since by assumption, every non-leaf vertex of $G$ is adjacent to a leaf, it is easy to see that the set of the leaves of $G$ and $T$ are the same. So the set of non-leaf vertices of them are also the same. For every non-leaf vertex $v$, denote the set of the leaves adjacent to it in $G$ by $L(v)$. Notice that $L(v)$ is the same as the set of the leaves adjacent to $v$ in $T$ and also $L(v) \not = \emptyset$. Since $A$ is a dominating set of $G$, for every non-leaf $v \in V(G)$, at least one of the two cases $v \in A$ or $L(v) \subseteq A$ happens. By these, it is easy to see that $A$ is a dominating set of $T$ too. The similar holds for $B$ and $C$ too.
	So far, we have shown that $A, B$ and $C$ are dominating sets of $T$ and every vertex of $T$ is either a leaf or adjacent to a leaf.
	Now by Step I, $T$ is a DM graph. Thus, by Proposition \ref{DMiff*everythreeDSisDS}, $*(A,B,C)$ is a dominating set of $T$. Hence, it is a dominating set of $G$ too. It follows the result.
	
	\vspace{1mm}
	
	$(\ref{MDScompMDS} \Rightarrow \ref{ENLOAL})$
	By Lemma \ref{Conectedcomponents}(\ref{AgraphisMDScompMDSiffeveryconnectedcomponentisMDScompMDS}), it would be enough to prove the statement for the case that $G$ is connected.
	We call a vertex $v$ of $G$ a {\em $*$-vertex} if there exists a maximal independent set (MIS), say $M$, and a vertex $w$ adjacent to $v$ such that $v \not \in M$ and $w \not \in M$.
	
	\vspace{0.75mm}
	
	{\bf \em Claim 1:}
	Every non-leaf $*$-vertex $v$ of $G$ is adjacent to a leaf.
	
	{\em Proof of Claim 1:} Since $v$ is $*$-vertex, there exist some MIS, say $M$, and a vertex $w$ adjacent to $v$ such that $v, w \not \in M$. Since every MIS is a minimal dominating set and also, by assumption, the complement of every minimal dominating set is a minimal dominating set, $M^c$ (i.e. $V(G)-M$) is a minimal dominating set.
	So every vertex in $M^c$ is either an isolated vertex of $\langle M^c \rangle$ (the induced subgraph on $M^c$), or has a private neighbor in $M$.
	Thus, since $v \in M^c$ and is not an isolated vertex in $\langle M^c \rangle$ (since it is adjacent to some vertex in $\langle M^c \rangle$, namely, $w$), it must have some private neighbor $s$ in $M$. Hence, $s$ is adjacent to $v$ and not adjacent to any other vertex in $M^c$. Also since $M$ is MIS, $s$ is not adjacent to any other vertex in $M$. It follows that $s$ is a leaf vertex in $G$. So $v$  is adjacent to a leaf vertex.
	\hfill {\em Claim 1} $\square$
	
	\vspace{0.75mm}
	
	{\bf \em Claim 2:} If a vertex $v$ is not a $*$-vertex in $G$, then every vertex of $G$ is either adjacent to all vertices in $N(v)$ or is not adjacent to any of them.
	
	{\em Proof of Claim 2:} Let $v$ be a not $*$-vertex and assume for contradiction that there is some vertex $z$ adjacent to a vertex $w_1 \in N(v)$ and non-adjacent to another vertex $w_2 \in N(v)$. Let $M$ be a maximal independent set of $G$ with $\{z,w_2\} \subseteq M$ (to find such $M$, it is enough to extend $\{z,w_2\}$ to a maximal independent set).
	Since $M$ is an independent set containing $w_2$, and $w_2$ is adjacent to $v$, $M$ does not contain $v$.
	Similarly, since $M$ contains $z$ and $w_1$ is adjacent to $z$, $M$ does not contain $w_1$.
	Therefore, $v$ is a vertex outside of the maximal independent set $M$ and is adjacent to the vertex $w_1 \not \in M$. It follows that $v$ is a $*$-vertex which is a contradiction.
	\hfill {\em Claim 2} $\square$
	
	\vspace{0.75mm}
	
	{\bf \em Claim 3:} 
	Assume that $v$ is a non-leaf vertex of $G$ which is not adjacent to any leaf. Then, $v$ is a $*$-vertex.
	
	{\em Proof of Claim 3:}
	Assume for contradiction that $v$ is a vertex of degree at least $2$ which is not adjacent to any leaf and is not $*$-vertex.
	So by above Claim 2, every vertex of $G$ is either adjacent to all vertices in $N(v)$ or is not adjacent to any of them.
	In particular, there is no edge between any two vertices in $N(v)$ since no vertex has a loop.
	Also, fixing some $w_0 \in N(v)$, we have
	
	$$\forall w \in N(v), \ N(w)=N(w_0). \ \ \ \ \  (1)$$
	Note that since $d(v) \geqslant 2$, we have $|N(v)| \geqslant 2$.
	Since $v$ is not adjacent to any leaf, every vertex in $N(v)$ has degree at least $2$. In particular, $|N(w_0)|\geqslant 2$.
	Hence, since there is no edge between any two vertices in $N(v)$, there must exists some $z \in V(G)-N[v]$ which is adjacent to $w_0$.
	So, by Claim 2, $z$ is adjacent to all vertices in $N(v)$. Thus, $d(z) \geqslant 2$ and $z$ is not a leaf  vertex. Similarly, every other vertex in $V(G)-N[v]$ which is adjacent to $w_0$ is not a leaf vertex.
	Therefore, $w_0$ is a non-leaf which is not adjacent to any leaf.
	So, by Claim 1, $w_0$ is not a $*$-vertex.
	Consequently, by Claim 2, every vertex of $G$ is either adjacent to all vertices in $N(w_0)$ or is not adjacent to any of them.
	Note that $v \in N(w_0)$. It follows that
	$$\forall r \in N(w_0), \ N(r)=N(v). \ \ \ \ \  (2)$$
	
	We want to show that $G$ is a complete bipartite graph.
	By using $(1)$ and $(2)$ above, it is not hard to verify that every vertex adjacent to some vertex in $N(v) \cup N(w_0)$ is already in $N(v) \cup N(w_0)$.
	Therefore, since $G$ is connected, we must have $V(G)=N(v) \cup N(w_0)$.
	On the other hand, we have $N(v) \cap N(w_0)=\emptyset$
	since otherwise if there is $x \in N(v) \cap N(w_0)$, then again by $(1)$ and $(2)$, we would have $N(x)=N(w_0)$ and $N(x)=N(v)$. Therefore, we would have $v \in N(w_0)=N(v)$ which would follow that $v \in N(v)$ which would be a contradiction. So $N(v) \cap N(w_0)=\emptyset$.
	Moreover, since, by (2), every two vertices in $N(w_0)$ must have the same open neighborhoods, there is no edge between any two vertices in $N(w_0)$. Also, as was shown above, there is no edge between any two vertices in $N(v)$.
	Moreover, by (1) and (2), obviously every vertex in $N(v)$ is adjacent to every vertex in $N(w_0)$ and vice versa.
	Putting together, $G$ is a complete bipartite graph with parts $A:=N(v)$ and $B:=N(w_0)$ where we remind from above that $|A| \geqslant 2$ and $|B| \geqslant 2$.
	
	Take $a \in A$ and $b \in B$.
	Obviously, $\{a,b\}$ is a minimal dominating set of $G$. So by assumption, $V(G)-\{a,b\}$
	must be a minimal dominating set. But since $G$ is complete bipartite with parts of size at least $2$, it is easily seen that we must have $|A|=|B|=2$. It implies that $G=C_4$ which is a contradiction.
	\hfill {\em Claim 3} $\square$
	
	\vspace{0.75mm}
	
	Now if $G$ has some non-leaf vertex $v$ which is not adjacent to any leaf, then by Claim 3, $v$ must be a $*$-vertex. So by Claim 1, $v$ must be adjacent to some leaf which is a contradiction. It follows that every vertex is either a leaf or adjacent to a leaf.
	
	\vspace{2mm}
	
	$(\ref{ENLOAL} \Rightarrow \ref{MDScompMDS})$
	For every non-leaf vertex $v$, denote the set of the leaves adjacent to $v$ by $L(v)$.
	It is not hard to verify that a subset $M \subseteq V(G)$ is a minimal dominating set of $G$ if and only if for every non-leaf $v \in V(G)$, exactly one of the two following cases takes place:
	
	- $v \in M$ and $L(v) \cap M=\emptyset$,
	
	- $v \not \in M$ and $L(v) \subseteq M$.
	
	Using this, it is not hard to see that for each minimal dominating set $M$, its complement set $V(G)-M$ is also a minimal dominating set. It follows the result.
	\hfill $\square$
\end{proof}

\vspace{1mm}

The following two immediate corollaries of Theorem \ref{thmDMiffMDScompMDSiffENLOAL}, characterize DM graphs (possibly with isolated vertices) and also summarize the characterization of MDScoMDS graphs.

\begin{cor}\label{characterizationofDMgraphs}
	Let $G$ be a graph (possibly with isolated vertices).
	Then, $G$ is DM if and only if every vertex of $G$ is either an isolated vertex, or a leaf, or is adjacent to a leaf.
\end{cor}
\begin{proof}
	Note that the graph $K_1$ is DM.
	Also it is not hard to apply Proposition \ref{DMiff*everythreeDSisDS} and verify that the graph $C_4$ is not a DM graph.
	Combining with Theorem \ref{thmDMiffMDScompMDSiffENLOAL} and Lemma \ref{Conectedcomponents}(\ref{AgraphisDMiffeveryconnectedcomponentisDM}) yields the result.
	\hfill $\square$
\end{proof}

\begin{cor}\label{characterizationofMDScoMDSgraphs}
	Let $G$ be a graph.
	Then, $G$ is MDScoMDS if and only if either $G$ is isomorphic to $C_4$ or every vertex of $G$ is either a leaf or adjacent to a leaf.
\end{cor}
\begin{proof}
	It is not hard to verify that $C_4$ is a MDScoMDS graph.
	Also note that if a graph has isolated vertex, then it can not be MDScoMDS.
	Combining these with Theorem \ref{thmDMiffMDScompMDSiffENLOAL} yields the result.
	\hfill $\square$
\end{proof}

\vspace{1mm}

The following corollary of Theorem \ref{thmDMiffMDScompMDSiffENLOAL},
makes a relation between median property of dominating graphs,
the notion of inverse domination number (see Definition \ref{dfninversedominationnumberofgraph})
and graphs $G$ with the property $\gamma(G)+\gamma^{-1}(G)=|V(G)|$.
Note that the class of graphs with property $\gamma(G)+\gamma^{-1}(G)=|V(G)|$ was considered in \cite{domkeDunbarMarkusTheinversedominationnumberofagraph},
and the following statement
clarifies the relationship between that class and graphs studied in this paper. 

\begin{cor}\label{inDMgraphsinversedominationnumberisequalnminusdominationnumber}
	Let $G$ be a graph with no isolated vertex and let $n:=|V(G)|$.
	If $G$ satisfies either of the
	conditions in the clauses of Theorem \ref{thmDMiffMDScompMDSiffENLOAL}, then $\gamma(G)+\gamma^{-1}(G)=n$.
	However, the converse does not necessarily hold.
\end{cor}
\begin{proof}
	If $G \simeq C_4$, then $G$ satisfies only condition \ref{MDScompMDS} (being MDScoMDS) and also it is fairly easy to see that in this case $\gamma(G)+\gamma^{-1}(G)=n$. So we may assume that $G \not \simeq C_4$ and in this case all three conditions of Theorem \ref{thmDMiffMDScompMDSiffENLOAL} would be equivalent for $G$. Assume that $G$ satisfies either of them.
	Let $A \subseteq V(G)$ be a smallest dominating set of $G$ whose complement contains a minimum dominating set, say $B$. Hence, by definition of inverse domination number, we have $\gamma^{-1}(G)=|A|$.
	Also obviously, $\gamma(G)=|B|$.
	Since $B$ is a minimum dominating set, $B$ is a minimal dominating set too, and since $G$ satisfies condition \ref{MDScompMDS} of Theorem \ref{thmDMiffMDScompMDSiffENLOAL}, the set $V(G)-B$ must be a minimal dominating set. But since $A \subseteq V(G)-B$ and $A$ is a dominating set, we must have $A=V(G)-B$. It follows that
	$\gamma^{-1}(G)=|A|=|V(G)-B|=n-\gamma(G)$.
	Therefore, $\gamma(G)+\gamma^{-1}(G)=n$.
	
	\vspace{1mm}
	
	Now we give an example to show that the converse does not necessarily hold.
	Let $G$ be the graph on $7$ vertices $\{a,b,c,d,e,f,g\}$ consisting of a cycle $C_3$ on vertices $a,b,c$ and four other edges $ad, ae, bf$ and $bg$.
	It is not hard to see that the only minimum dominating set of $G$ is the set $\{a,b\}$ and also the inverse domination number is equal to $5$ which is realized  by the set $\{c,d,e,f,g\}$.
	So, we have $\gamma(G)+\gamma^{-1}(G)=7$. But on the other hand, since $c$ is a vertex which is not a leaf and is not adjacent to any leaf, $G$ does not satisfy condition \ref{ENLOAL}, and therefore, none of the equivalent conditions of Theorem \ref{thmDMiffMDScompMDSiffENLOAL}.
	\hfill $\square$
\end{proof}

\vspace{1mm}

The following corollary shows that every graph $G$ can be extended to a DM and MDScoMDS graph which is not much bigger than $G$ and contains $G$ as an induced subgraph.
\begin{cor}
	For every graph $G$, there exists a graph $H$ with $|V(H)| \leqslant 2|V(G)|$ which is a DM and MDScoMDS graph and has $G$ as an induced subgraph.
\end{cor}
\begin{proof}
	Let $H$ be $G \odot K_1$, the corona product of $G$ by $K_1$. This product adds and connect a new leaf to any vertex of $G$ and so $|V(H)| = 2|V(G)|$.
	Note that if $G$ itself has some leaves, then we can do better and instead of the above corona product, define $H$ as follows. Corresponding to every non-leaf vertex of $G$ which is not adjacent to any leaf, add a new vertex to $G$ and connect it only to that vertex. Letting $H$ to be the obtained graph, we have $|V(H)| < 2|V(G)|$ in this case.
	
	In any case, $H$ has $G$ as an induced subgraph and $|V(H)| \leqslant 2|V(G)|$. Moreover, by Theorem \ref{thmDMiffMDScompMDSiffENLOAL}, $H$ is a DM graph and also MDScoMDS.
	\hfill $\square$
\end{proof}

\subsection{Dominating graphs and partial cubes}

The following result establishes the relationship between dominating graph of graphs and partial cubes.
Also, it provides us with some novel infinite sources of examples of partial cubes
since, as we will see, the dominating graph of every graph gives rise to a partial cube.
As mentioned in the introduction, finding new sources of examples of partial cubes 
has always been a topic of interest and was considered extensively in the literature.
In the following statement, we also give some examples to show that not all partial cubes or median graphs are isomorphic to the dominating graph of a graph.

\begin{thm}\label{Dominatinggraphsarepartialcube}
	The dominating graph $D_{|G|}(G)$ of any graph $G$ is a partial cube. On the other hand, there are partial cubes (and even median graphs) that are not isomorphic to the dominating graph of any graph.
\end{thm}
\begin{proof}
	For simplicity, denote the dominating graph $D_{|G|}(G)$ by $H$.
	Let $D_1$ and $D_2$ be two vertices of $H$ (so they are dominating sets of $G$).
	We construct a path $P$ in $H$ between $D_1$ and $D_2$ as follows:
	\begin{enumerate}
		\item{Start from $D_1$ and add one by one (according to an arbitrary ordering) every vertex in $D_2 - D_1$ until reaching to $S_1:=D_1 \cup D_2$. Since every set in this process contains $D_1$, they are all dominating sets.}
		
		\item{Remove one by one (according to an arbitrary ordering) every vertex in $D_1-D_2$ from $S_1$ 
			until reaching to $D_2$. Since every set in this process contains $D_2$, they are all dominating sets.}
	\end{enumerate}

	Since every set obtained in above steps of moving from $D_1$ to $D_2$ on the path $P$ is a dominating set, $P$ is a path in $H$ connecting $D_1$ to $D_2$. Also it is easily seen that the length of $P$ is $|D_1 \triangle D_2|$.
	It follows that $H$ is a partial cube.
	
	Now we show that there are partial cubes that are not isomorphic to the dominating graph of any graph.
	It is easy to see that any graph $K_{1,n}$ is a partial cube and even more, is a median graph.
	We show that for $n \geqslant 3$, the graph $K_{1,n}$ is not isomorphic to the dominating graph of any graph.
	
	Assume for contradiction that for some $n\geqslant 3$, the graph $K_{1,n}$ is isomorphic to $H$ with some isomorphism $f:V(H) \rightarrow V(K_{1,n})$ for some graph $G$ where we again use $H$ for denoting $D_{|G|}(G)$.
	Let $c$ be the central vertex of $K_{1,n}$.
	Also let $V:=V(G)$.
	Since $V(G)$ is a dominating set of $G$, obviously $V$ is a vertex of $H$. Let $a:=f(V)$.
	
	\vspace{1mm}
	
	{\bf \em Claim :} $a=c$.
	
	{\em Proof of Claim:}
	Assume for contradiction that $a \not =c$. Then, $a$ is a leaf vertex of $K_{1,n}$ and so $V$ is a leaf vertex of $H$. Let $D$ be the vertex in $H$ such that $f(D)=c$.
	Since $a$ is adjacent to $c$ in $K_{1,n}$, $V$ must be adjacent to $D$ in $H$. It follows that there is $v \in G$ such that $D =V- \{v\}$. 
	Let $b$ be another (different from $a$) leaf vertex adjacent to $c$ in $K_{1,n}$ and assume that $D' \in V(H)$ be such that $f(D')=b$.
	So $D'$ is adjacent to $D$ in $H$.
	Since $D=V-\{v\}$, it is not hard to see that $D'=D -\{w\}$ for some $w \in D$.
	Thus, $D'=V-\{v,w\}$.
	Since $D'$ is a dominating set of $G$, the set $D'':=D' \cup \{v\}=V-\{w\}$ is also a dominating set of $G$ and so $D'' \in V(H)$. Also obviously $D''$ is different from $D$ and $V$ and moreover, it is adjacent to $V$ in $H$ since $V=D'' \cup \{w\}$. Now $V$ is adjacent to both vertices $D$ and $D''$ in $H$ which is a contradiction since $V$ was a leaf vertex of $H$. Therefore, $a=c$.
	\hfill {\em Claim} $\square$
	
	\vspace{1mm}
	
	Since by above claim $f(V)=c$, and $c$ is adjacent to every other vertex in $K_{1,n}$, $V$ must be also adjacent to every other vertex in $H$. It follows that every dominating set of $G$ except $V$ is of size $|G|-1$. Thus, no subset of size $|G|-2$ of $V(G)$ is a dominating set of $G$. Hence, for every $z,t \in V(G)$, at least one of $z$ or $t$, say $z$, is not adjacent to any vertex in $V(G)-\{z,t\}$. So, either $z$ is an isolated vertex of $G$ or $z$ and $t$ are adjacent to each other.
	Since the same happens for every two vertices of $G$, then every two non-isolated vertices of $G$ are adjacent to each other.
	Therefore, $G$ must be a union of a complete graph and some isolated vertices.
	But it is not difficult to verify that in such graphs, the dominating graph $H$ is not isomorphic to $K_{1,n}$ for any $n \geqslant 3$. This is a contradiction. It follows that for any $n \geqslant 3$, the graph $K_{1,n}$ is not isomorphic to the dominating graph of any graph.
	Therefore, there are partial cubes and median graphs which are not isomorphic to the dominating graph of any graph.
	\hfill $\square$
\end{proof}

\subsection{Discussion}

In relation with the main Theorems \ref{thmDMiffMDScompMDSiffENLOAL} and \ref{Dominatinggraphsarepartialcube}
(in which the dominating graph $D_{|G|}(G)$ of a graph $G$ was under consideration), it is worth to mention that
for $k < |G|$, the graph $D_k(G)$ is not necessarily partial cube (and therefore in that case not median either).
More precisely, if $k <|G|$, then the graph $D_k(G)$ might not even be connected, in which case can not be partial cube or median. We give an example. Let $G$ be a graph with at least one edge.
So, by Remark \ref{everygraphwithatleastoneedgehasatleasttwoMDS}, $G$ has at least two distinct minimal dominating sets. Let $M$ be a largest minimal dominating set of $G$ and let $k:=|M|$. Then, $D_k(G)$ has at least two vertices. We claim that $D_k(G)$ is not connected. The reason is that $M$, as a vertex of $D_k(G)$, is not adjacent to any other vertex in $D_k(G)$ since every vertex adjacent to $M$ must either have exactly one vertex less than $M$ or exactly one vertex more. It is easy to see that since $M$ is minimal dominating set with $|M|=k$, both cases are impossible.

\vspace{3mm}


\def\germ{\frak} \def\scr{\cal} \ifx\documentclass\undefinedcs
  \def\bf{\fam\bffam\tenbf}\def\rm{\fam0\tenrm}\fi 
  \def\defaultdefine#1#2{\expandafter\ifx\csname#1\endcsname\relax
  \expandafter\def\csname#1\endcsname{#2}\fi} \defaultdefine{Bbb}{\bf}
  \defaultdefine{frak}{\bf} \defaultdefine{=}{\B} 
  \defaultdefine{mathfrak}{\frak} \defaultdefine{mathbb}{\bf}
  \defaultdefine{mathcal}{\cal} \defaultdefine{implies}{\Rightarrow}
  \defaultdefine{beth}{BETH}\defaultdefine{cal}{\bf} \def\bbfI{{\Bbb I}}
  \def\mbox{\hbox} \def\text{\hbox} \def\om{\omega} \def\Cal#1{{\bf #1}}
  \def\pcf{pcf} \defaultdefine{cf}{cf} \defaultdefine{reals}{{\Bbb R}}
  \defaultdefine{real}{{\Bbb R}} \def\restriction{{|}} \def\club{CLUB}
  \def\w{\omega} \def\exist{\exists} \def\se{{\germ se}} \def\bb{{\bf b}}
  \def\equivalence{\equiv} \let\lt< \let\gt>

\end{document}